\documentclass{amsart}

\usepackage{amsmath,amsthm}
\usepackage{amssymb}
\usepackage{mathtools}
\usepackage{amscd}
\usepackage{url}

\newtheorem{thm}{Theorem}[section]
\newtheorem{prop}[thm]{Proposition}
\newtheorem{lem}[thm]{Lemma}
\theoremstyle{definition}
\newtheorem{rem}{Remark}
\newtheorem{defi}{Definition}[section]
\newtheorem{assump}{Assumption}[section]

\newcommand{\Z}{\mathbb{Z}}
\newcommand{\Q}{\mathbb{Q}}
\newcommand{\F}{\mathbb{F}}
\newcommand{\Comp}{\mathbb{C}}
\newcommand{\set}[2]{\{#1\ |\ #2\}}
\newcommand{\Order}{\mathcal{O}}
\newcommand{\Orderl}[2]{\Order_{#2}^{{(#1)}}}
\newcommand{\ia}{\mathfrak{a}}
\newcommand{\ib}{\mathfrak{b}}
\newcommand{\iq}{\mathfrak{q}}
\newcommand{\ipe}{\mathfrak{p}}
\newcommand{\iPe}{\mathfrak{P}}
\newcommand{\End}{\mathrm{End}}

\newcommand{\orientdef}[1]{\varphi\circ\iota(#1)\circ\hat{\varphi}}
\newcommand{\cl}{\mathcal{C}\hspace{-1.8pt}\ell}
\newcommand{\SSpr}[1]{\mathrm{SS}^{\mathrm{pr}}_{#1}(p)}
\newcommand{\Ell}{\mathcal{E}\hspace{-1pt}\ell\hspace{-1pt}\ell}
\newcommand{\Emod}[2]{\widetilde{#1}}
\newcommand{\Incmod}[3]{[#1]_{\Emod{#2}{#3}}}

\newcommand{\cdotsp}{\ \cdot\ }
\newcommand{\setj}[1]{\mathcal{J}_{#1}}
\newcommand{\Gal}[2]{\mathrm{Gal}(#1/#2)}
\newcommand{\redmap}[1]{\rho}

\newcommand{\SSprm}[1][\Order]{\redmap{\ipe}(\Ell(#1))}
\newcommand{\Graph}{G_\ell(K, p)}
\newcommand{\Legendre}[2]{\left(\frac{#1}{#2}\right)}

\begin{document}
    \title{On oriented supersingular elliptic curves}

    \author{Hiroshi Onuki}
    \address{Derparment of Mathematical Informatics, The University of Tokyo,
        7-3-1 Hongo, Bunkyo-ku, Tokyo 113-8656, Japan}
    \curraddr{}
    \email{onuki@mist.i.u-tokyo.ac.jp}
    \thanks{This work was supported by JST CREST Grant Number JPMJCR14D6, Japan.}

    \date{}
    \dedicatory{}

    \begin{abstract}
        We revisit theoretical background on OSIDH (Oriented Supersingular Isogeny Diffie-Hellman protocol),
        which is an isogeny-based key-exchange protocol proposed by Col\`o and Kohel at NutMiC 2019.
        We give a proof of a fundamental theorem for OSIDH.
        The theorem was stated by Col\`o and Kohel without proof.
        Furthermore,
        we consider parameters of OSIDH,
        give a sufficient condition on the parameters for the protocol to work,
        and estimate the size of the parameters for a certain security level.
    \end{abstract}

    \maketitle

    \section{Introduction}
Isogeny-based cryptography is based on hardness of the isogeny problem,
that is a problem to find an isogeny between two given elliptic curves.
The isogeny problem is considered to be hard even if one uses a quantum computer.
Therefore, isogeny-based cryptography is one of the candidates for post quantum cryptography.
The first isogeny-based cryptosystem was proposed by Couveignes \cite{cou} in 1997.
But his work was not published and posted on ePrint in 2006.
The same result was rediscovered by Rostovtsev and Stolbunov \cite{RS,Stol}.
Their cryptosystem is a key-exchange protocol using isogenies between ordinary elliptic curves.
The isogeny problem on supersingular elliptic curves
first appeared
in hash functions proposed by Charles, Goren, and Lauter \cite{CGL2006}.
In 2011, Jao and De Feo \cite{JF} proposed an isogeny-based exchange protocol
using supersingular elliptic curves, named SIDH (Supersingular Isogeny Diffie-Hellman).
Castryck, Lange, Martindale, Panny, and Renes \cite{CSIDH}
proposed another isogeny-based key-exchange protocol CSIDH (Commutative SIDH).
CSIDH uses an action of an ideal class group on a set of classes of supersingular elliptic curves.

Currently, many researches focus on the protocols using supersingular elliptic curves
due to their efficiency.
Even after optimizations by De Feo, Kieffer, and Smith \cite{FKS},
the protocol using ordinary elliptic curves is much slower than SIDH and CSIDH.
The endomorphism ring of a supersingular elliptic curve is isomorphic to a maximal order of a quaternion algebra.
The isogeny problem is closely related the structure of the endomorphism ring.
Therefore, it is important for crytanalysis to study the endomorphism ring.
Indeed, there are several researches on this topic.
For example, see \cite{kohel1996,kohel:hal-01257092,GPST,EHLMP,CPV2019,love2019supersingular,eisentraeger2020computing}.

In 2019,
Col\`o and Kohel \cite{CK19}
proposed a new isogeny-based key-exchange protocol using
isogenies between supersingular elliptic curves, named OSIDH (Oriented Supersingular Isogeny Diffie-Hellman protocol).
OSIDH uses an inclusion
$$\Order \hookrightarrow \End(E),$$
where $\Order$ is an order of an imaginary quadratic field,
and $E$ is a supersingular elliptic curve over a finite field.
Col\`o and Kohel stated that
the ideal class group $\cl(\Order)$ acts freely and transitively on
a set of equivalence classes of supersingular elliptic curves
which have the above inclusion.
It can be seen as a generalization of the result by Waterhouse \cite{Waterhouse69}.
However, they did not give any proof of it.

In this paper,
we show that their claim has to be slightly modified
and give a proof of a modified theorem.
Furthermore, we show that a method to calculate the group action
proposed by Col\`o and Kohel does not work in some cases,
and give a sufficient condition that the method works.
Under this condition,
we estimate the size of parameters of OSIDH for a certain security level.

    \section{Notation}
Throughout this paper,
we use the following notation.

We fix an algebraic closure of a finite field of characteristic $p$
and denote it by $k$.

We let $K$ be an imaginary quadratic field, $\Order_K$ the ring of integer of $K$,
and $\Order$ an order in $K$.
There exists a unique positive integer $c$ such that
$\Order = \Z + c\Order_K$.
We call the number $c$ the {\it conductor} of $\Order$.
We denote the class group of $\Order$ by $\cl(\Order)$
and the class number of $\Order$ by $h(\Order)$.
For $\alpha \in K$,
we denote the complex conjugate of $\alpha$ by $\bar{\alpha}$.

For an elliptic curve $E$,
we denote the identity element of $E$ by $0_E$,
the $j$-invariant of $E$ by $j(E)$,
and the endomorphism ring of $E$ by $\End(E)$.
We define $\End^0(E) \coloneqq \End(E) \otimes_\Z \Q$.
For an isogeny $\varphi$,
we denote the dual isogeny of $\varphi$ by $\hat{\varphi}$.

For a set $S$, we denote the cardinality of $S$ by $\#S$.
For a group $G$ and $g \in G$,
we denote the subgroup of $G$ generated by $g$ by $\langle g \rangle$.

    \section{Oriented supersingular elliptic curves}
In this section,
we define orientations on elliptic curves over $k$
and prove that an ideal class group acts freely and transitively on
a set of equivalence classes of oriented supersingular elliptic curves.

\subsection{Orientations}
We recall definitions about orientations on elliptic curves,
which are given in \cite{CK19}.

\begin{defi}
    A {\it $K$-orientation} on an elliptic curve $E/k$ is
    a ring homomorphism
    $$\iota: K \hookrightarrow \End^0(E).$$

    A $K$-orientation on $E$ is an {\it $\Order$-orientation}
    if $\iota(\Order) \subseteq \End(E)$.
    An $\Order$-orientation is {\it primitive}
    if $\iota(\Order) = \End(E) \cap \iota(K)$.
    If $\iota$ is a $K$-orientation on $E$
    (resp. primitive $\Order$-orientation),
    a pair $(E, \iota)$ is called
    a {\it $K$-oriented} (resp. {\it primitive} {\it $\Order$-oriented}) {\it elliptic curve}.
\end{defi}

Let $(E, \iota)$ be a $K$-oriented elliptic curve and $\alpha \in K$
such that $\iota(\alpha) \in \End(E)$.
We have $\iota(\bar{\alpha}) = \widehat{\iota(\alpha)}$.
The degree and the trace of $\iota(\alpha)$ are equal to the norm and the trace of $\alpha$, respectively.
In particular, the element $\alpha$ is integral over $\Z$.

\begin{defi}
    Let $(E, \iota)$ be a $K$-oriented elliptic curve
    and $\varphi: E \to F$ an isogeny of degree $\ell$.
    We define a $K$-orientation $\varphi_*(\iota)$ on $F$ by
    \begin{equation*}
        \varphi_*(\iota)(\alpha) = \frac{1}{\ell}\orientdef{\alpha}
        \quad \mbox{for}\ \alpha \in K.
    \end{equation*}
    
    Given two $K$-orientations $(E, \iota_E)$ and $(F, \iota_F)$,
    an isogeny $\varphi: E \to F$ is {\it $K$-oriented}
    if $\varphi_*(\iota_E) = \iota_F$.
    We denote this by
    $\varphi: (E, \iota_E) \to (F, \iota_F)$.
\end{defi}

For a $K$-oriented isogeny $\varphi : (E, \iota_E) \to (F, \iota_F)$,
let $\Order = \End(E) \cap \iota_E(K)$ and $\Order' = \End(F) \cap \iota_F(K)$
so that $\iota_E$ is a primitive $\Order$-orientation
and $\iota_F$ is a primitive $\Order'$-orientation.
We say that $\varphi$ is
{\it horizontal} if $\Order = \Order'$,
{\it ascending} if $\Order \subsetneq \Order'$,
and {\it descending} if $\Order \supsetneq \Order'$.

\begin{defi}
    A $K$-oriented isogeny $\varphi: (E, \iota_E) \to (F, \iota_F)$
    is a {\it $K$-oriented isomorphism}
    if there exists a $K$-oriented isogeny $\psi: (F, \iota_F) \to (E, \iota_E)$
    such that
    $\psi\circ\varphi= \mathrm{id}_E$ and
    $\varphi\circ\psi = \mathrm{id}_F$ as maps.
    If this happens, we say that $(E, \iota_E)$ and $(F, \iota_F)$ are {\it $K$-isomorphic}
    and write $(E, \iota_E) \cong (F, \iota_F)$.
\end{defi}

Note that a $K$-isomorphism $\varphi$ and its inverse $\varphi^{-1}$ are horizontal.

Let $(E, \iota)$ be a $K$-oriented elliptic curve over $k$.
There is the $p$-th power Frobenius map
$\phi_p: E \to E^{(p)}$,
where $E^{(p)}$ is the elliptic curve
obtained from $E$ by raising each coefficients of $E$ to the $p$-th power.
Then we denote $(\phi_p)_*(\iota)$ by $\iota^{(p)}$.
It can be easily checked that
$K$-oriented isogeny
$\phi_p: (E, \iota) \to (E^{(p)}, \iota^{(p)})$
is horizontal.
Furthermore, if $E$ is supersingular then
$(E, \iota)$ is $K$-isomorphic to $((E^{(p)})^{(p)}, (\iota^{(p)})^{(p)})$,
since $E$ is isomorphic to an elliptic curve defined over $\F_{p^2}$
whose endomorphism ring is also defined over $\F_{p^2}$.

We denote the set of primitive $\Order$-oriented supersingular elliptic curves up to $K$-isomorphism
by $\SSpr{\Order}$ as in \cite{CK19}.
We write a $K$-isomorphism class by the same symbol as one of its representatives for brevity.
Note that we can always take a representative of a class in $\SSpr{\Order}$ defined over $\F_{p^2}$.

In \cite{CK19},
it is claimed that $\cl(\Order)$ acts freely and transitively on $\SSpr{\Order}$.
However, rigorously it is not correct.
We should slightly modify their claim.
In the following, we explain this by showing a counter example.

Let $E$ be an elliptic curve over $k$ defined by $y^2 = x^3 + x$.
As is well known, $\End(E)$ contains a subring isomorphic to $\Z[i]$, 
where $i$ is a square root of $-1$ in $\Comp$.
We assume $p \equiv 3 \pmod{4}$. Then $E$ is supersingular.
Let $a$ be a square root of $-1$ in $\F_{p^2}$.
Then there are two orientations
\begin{eqnarray*}
    &&\iota:~ \Q(i) \to \End^0(E),\quad i \mapsto ((x, y) \mapsto (-x, ay)),\\
    &&\iota':~ \Q(i) \to \End^0(E),\quad i \mapsto ((x, y) \mapsto (-x, -ay)),
\end{eqnarray*}
and two primitive $\Z[i]$-oriented elliptic curves $(E, \iota)$ and $(E, \iota')$.
It is easy to show that $(E, \iota)$ and $(E, \iota')$ are not $K$-isomorphic
by checking all the automorphisms of $E$ (there are exactly four automorphisms).
Therefore, there exists at least two
$\Q(i)$-isomorphism classes of primitive $\Z[i]$-oriented supersingular elliptic curves.
On the other hand, the class number of $\Z[i]$ is one,
so the class group of $\Z[i]$ never acts transitively on the set of these classes.

To fix the claim in \cite{CK19},
we consider reductions of elliptic curves over number fields in the next subsection.

    \subsection{Reductions}
Let $L$ be a number field containing $K$ and
$E$ an elliptic curve over $L$ with $\End(E) \cong \Order$.
Let $[\cdotsp]_E: \Order \to \End(E)$ be an isomorphism such that $(E, [\cdotsp]_E)$ is normalized,
i.e., for any invariant differential $\omega$ on $E$,
$$([\alpha]_E)^*\omega = \alpha\omega,\quad \mbox{for all }\alpha \in \Order.$$
(See II.1 in \cite{silverman1994advanced}.)

Let $\ipe$ be a prime ideal of $L$ lying above $p$ at which $E$ has a good reduction.
A pair $(E, [\cdotsp]_E)$ determines a $K$-oriented elliptic curve $(\Emod{E}{\ipe}, \Incmod{\cdotsp}{E}{\ipe})$
by the reduction modulo $\ipe$,
where $\Incmod{\cdotsp}{E}{\ipe} : K \to \End^0(\Emod{E}{\ipe})$ is defined by
$$\Incmod{\alpha}{E}{\ipe} = [\alpha]_E \bmod{\ipe} \mbox{ for all } \alpha \in \Order.$$
For two isomorphic elliptic curves $E$ and $E'$ over $L$ and an isomorphism $\lambda: E \to E'$,
it holds that
$[\cdotsp]_{E'} = \lambda\circ[\cdotsp]_E\circ\lambda^{-1}$.
Therefore, 
the $K$-isomorphism class of the reduction is determined by the isomorphism class of an elliptic curve over a number field.
The following lemma states properties of these reductions.

\begin{lem}\label{lem:Lang}
    Let $E$ be an elliptic curve over a number field $L$ containing $K$
    with $\End(E) \cong \Order$,
    and $\ipe$ a prime ideal of $L$ lying above $p$
    such that $E$ has a good reduction at $\ipe$.
    Then the reduction curve $\Emod{E}{\ipe}$ modulo $\ipe$ is supersingular
    if and only if $p$ does not split in $K$.
    Furthermore, let $c$ be the conductor of $\Order$ and write $c = p^rc_0$, where $p \nmid c_0$.
    Then
    $$\End(\Emod{E}{\ipe}) \cap \Incmod{K}{E}{\ipe} = \Incmod{\Z + c_0\Order_K}{E}{\ipe}.$$
\end{lem}

\begin{proof}
    See Theorem 12 in Chapter 13 in \cite{lang1987elliptic}.
    The statement about the endomorphism ring can be proved in a similar way to that of ordinary elliptic curves.
\end{proof}

This lemma shows that
if $p$ does not split in $K$ and does not divide the conductor of $\Order$,
the reduction $(\Emod{E}{\ipe}, \Incmod{\cdotsp}{E}{\ipe})$
is a primitive $\Order$-oriented supersingular elliptic curve.
From this lemma,
we obtain the following proposition.

\begin{prop}\label{prop:num_of_SS}
    The set $\SSpr{\Order}$ is not empty if and only if
    $p$ does not split in $K$ and does not divide the conductor of $\Order$.
\end{prop}

\begin{proof}
    First, we assume that $p$ does not split in $K$ and does not divide the conductor of $\Order$.
    There exists an elliptic curve $E$ over a number field $L$ with $\End(E) \cong \Order$.
    Since the $j$-invariant of $E$ is an algebraic integer (Theorem II.6 in \cite{silverman1994advanced}),
    $E$ has a potential good reduction at every prime ideal (Proposition VII.5.5 in \cite{silverman2009arithmetic}).
    Therefore, we may assume that $E$ has a good reduction at a prime ideal of $L$ lying above $p$.
    By Lemma \ref{lem:Lang},
    the reduction $(\Emod{E}{\ipe}, \Incmod{\cdotsp}{E}{\ipe})$ modulo such a prime ideal
    is a primitive $\Order$-oriented supersingular elliptic curve.
    Therefore, $\SSpr{\Order}$ is not empty.

    Next, we consider the converse.
    Assume $\SSpr{\Order}$ is not empty and let $(F, \iota) \in \SSpr{\Order}$.
    Let $\theta \in \Order$ such that $\Order = \Z[\theta]$.
    The Deuring lifting theorem (Theorem 14 in Chapter 13 in \cite{lang1987elliptic})
    says that there exist an elliptic curve $E$ over a number field $L$,
    an endomorphism $\alpha \in \End(E)$,
    and a prime ideal $\ipe$ lying above $p$
    such that
    the reduction $\Emod{E}{\ipe}$ modulo $\ipe$ is isomorphic to  $F$ and $\alpha \bmod \ipe$
    corresponds to $\iota(\theta)$ under the isomorphism.
    Since $\alpha$ has the same degree and trace as $\iota(\theta)$,
    $\End(E)$ contains a subset isomorphic to $\Order$.
    Since the reduction map $\End(E) \to \End(F)$ is injective,
    we have $\End(E) \cong \Order$.
    By Lemma \ref{lem:Lang},
    $p$ does not split in $K$ and does not divide the conductor of $\Order$.
\end{proof}

Let $p$ does not split in $K$ and does not divide the conductor of $\Order$.
We denote the set of $j$-invariants of elliptic curves $E$ over $\Comp$ with
$\End(E) \cong \Order$ by $\setj{\Order}$.
Since all elements in $\setj{\Order}$ are algebraic integers,
an elliptic curve whose $j$-invariant is in $\setj{\Order}$
has a potential good reduction at any prime ideal.
Since $\setj{\Order}$ is finite,
we can take a number field $L$ and a prime ideal $\ipe$ of $L$ lying above $p$
such that
for all $j \in \setj{\Order}$,
there exists an elliptic curve over $L$ whose $j$-invariant is $j$
and that has a good reduction at $\ipe$.
Hereafter,
we fix a number field $L$ and a prime ideal $\ipe$ which have these properties,
and an inclusion from the residue field of $L$ modulo $\ipe$ to $k$.
We denote the set of isomorphism classes of elliptic curves $E$ over $L$
such that $j(E) \in \setj{\Order}$ and $E$ has a good reduction at $\ipe$
by $\Ell(\Order)$.
Then we obtain a map defined by the reduction modulo $\ipe$:
$$\redmap{\ipe} : \Ell(\Order) \to \SSpr{\Order},\quad
    E \to (\Emod{E}{\ipe}, \Incmod{\cdotsp}{E}{\ipe}).$$
We show that $\redmap{\ipe}$ is surjective up to the $p$-th power Frobenius map.

\begin{prop}\label{prop:surjection_to_SSpr}
    For all $(F, \iota) \in \SSpr{\Order}$, we have
    $$(F, \iota) \mbox{ or } (F^{(p)}, \iota^{(p)}) \in \redmap{\ipe}(\Ell(\Order)).$$
\end{prop}

\begin{proof}
    Let $(F, \iota) \in \SSpr{\Order}$.
    As in the proof of Proposition \ref{prop:num_of_SS},
    there exist a number field $L'$, a prime ideal $\ipe'$ of $L'$,
    and an elliptic curve $E$ over $L'$ with $\End(E) \cong \Order$
    such that $E$ has a good reduction at $\ipe'$ and
    the reduction $\Emod{E}{\ipe'}$ modulo $\ipe'$ is isomorphic to $F$.
    We can assume that $L'$ is a Galois extension over $\Q$.
    Then the induced inclusion $\Incmod{\cdotsp}{E}{\ipe'}$ is equal to $\iota$
    or it holds that $\Incmod{\alpha}{E}{\ipe'} = \iota(\bar{\alpha})$,
    for all $\alpha \in \Order$.

    Assume the latter holds.
    Let $G_{\ipe'}$ be the decomposition group of $\ipe'$, i.e.,
    $$G_{\ipe'} = \set{\sigma \in \Gal{L'}{\Q}}{\ipe'^{\sigma} = \ipe'}.$$
    Since $p$ does not split in $K$,
    there exists $\sigma \in G_{\ipe'}$ such that the restriction of $\sigma$ on $K$ is not trivial.
    Then, for $\alpha \in \Order$, we have $[\alpha]_{E^{\sigma}} = ([\bar{\alpha}]_{E})^{\sigma}$.
    Therefore, the reduction
    $(\Emod{E^{\sigma}}{\ipe'}, \Incmod{\cdotsp}{E^{\sigma}}{\ipe'})$ is $K$-isomorphic to
    $(F, \iota)$ or $(F^{(p)}, \iota^{(p)})$ (it is determined by the reduction of $\sigma$ modulo $\ipe'$).

    Consequently,
    we obtain a number field $L'$, a prime ideal $\ipe'$
    and an elliptic curve $E$ over $L'$ such that the reduction $(\Emod{E}{\ipe'}, \Incmod{\cdotsp}{E}{\ipe'})$
    is $K$-isomorphic to $(F, \iota)$ or $(F^{(p)}, \iota^{(p)})$.

    Let $M$ be a finite Galois extension over $K$ containing $L$ and $L'$.
    Let $\iPe$ and $\iPe'$ be prime ideals of $M$ lying above $\ipe$ and $\ipe'$, respectively.
    Since $K$ has only one prime ideal lying above $p$,
    there exists $\sigma \in \Gal{M}{K}$ such that $\iPe = \iPe'^{\sigma}$.
    Then we have $[\alpha]_{E^{\sigma}} = ([\alpha]_{E})^\sigma$, for all $\alpha \in \Order$.
    Therefore the reduction $(\Emod{E^\sigma}{\iPe}, \Incmod{\cdotsp}{E^{\sigma}}{\iPe})$ modulo $\iPe$
    is $K$-isomorphic to 
    the reduction $(\Emod{E}{\iPe'}, \Incmod{\cdotsp}{E}{\iPe'})$ modulo $\iPe'$
    or its image $(\Emod{E}{\iPe'}^{(p)}, (\Incmod{\cdotsp}{E}{\iPe})^{(p)})$ under the $p$-th power Frobenius map.
    Since $j(E^\sigma) \in \setj{\Order}$,
    there exists an elliptic curve $E'$ over $L$ such that $E' \cong E^\sigma$ and $E'$ has a good reduction at $\ipe$.
    Then we have $\redmap{\ipe}(E') = (F, \iota)$ or $(F^{(p)}, \iota^{(p)})$.
\end{proof}

    \subsection{Group action}\label{sec:group_action}

The first objective in this paper is to prove the following theorem.
This is a slightly modified version of Proposition 3.1 in \cite{CK19},
that is stated without proof.

\begin{thm}\label{thm:group_action}
    Let $K$ be an imaginary quadratic field such that $p$ does not split in $K$,
    and $\Order$ an order in $K$ such that $p$ does not divide the conductor of $\Order$.
    Then the ideal class group $\cl(\Order)$ acts freely and transitively on $\SSprm$.
\end{thm}

Before we prove this theorem, we define an action of $\cl(\Order)$ on $\SSprm$.
For $(E, \iota) \in \SSpr{\Order}$ and an integral ideal $\ia$ of $\Order$ prime to $p$,
we define the $\ia$-torsion subgroup by
$$E[\ia] = \bigcap_{\alpha \in \ia} \ker\iota(\alpha).$$
Then there are an elliptic curve $F$ and a separable isogeny $\varphi : E \to F$
with $\ker\varphi = E[\ia]$
(Proposition III.4.12 in \cite{silverman2009arithmetic}).
Furthermore,
if there are an elliptic curve $F'$ and a separable isogeny $\varphi' : E \to F'$
with $\ker\varphi' = E[\ia]$,
there exists an isomorphism $\lambda : F \to F'$ such that
$\varphi' = \lambda\circ\varphi$
(Corollary III.4.11 in \cite{silverman2009arithmetic}).
Therefore, we have $(F, \varphi_*(\iota))$ is $K$-isomorphic to $(F', \varphi'_*(\iota))$.
We denote the $K$-isomorphism class of $(F, \varphi_*(\iota))$ by $\ia * (E, \iota)$.
Then we have the following proposition.

\begin{prop}\label{prop:horizontal}
    Let $(E, \iota)$ be a primitive $\Order$-oriented elliptic curve,
    $\ia$ an integral ideal of $\Order$ prime to $p$.
    Then a $K$-oriented isogeny with kernel $E[\ia]$
    $$\varphi : (E, \iota) \to \ia * (E, \iota)$$
    is horizontal or ascending.
    Furthermore, if $\ia$ is invertible then $\varphi$ is horizontal.
\end{prop}

\begin{proof}
    Let $(E', \iota') = \ia * (E, \iota)$ and $\Order' = \End(E') \cap \iota'(K)$.

    If $\ia$ is divisible by some integer $n$,
    we have $\varphi = \varphi' \circ \iota(n)$, where $\varphi'$ is an isogeny with kernel $E[\ia/n]$.
    Since the multiplication by $n$ in $\End(E)$ is horizontal,  
    we can assume that $\ia$ is not divisible by any integer greater than 1,
    i.e., $E[\ia]$ is cyclic.

    Let $a$ be the absolute norm of $\ia$ and $x \in \Order$.
    By definition,
    \begin{equation*}
        \iota'(x) = \frac{1}{a}\orientdef{x}.
    \end{equation*}
    Therefore,
    $\Order \subseteq \Order'$
    if and only if $\orientdef{x} \in a\End(E')$ for all $x \in \Order$.
    Let $\{P, Q\}$ be a pair of points in $E$ generating $E[a]$ such that $P$ generates $E[\ia]$,
    and $P'$ a point in $E$ such that $aP' = P$.
    Then it can be easily checked that $\varphi(P')$ and $\varphi(Q)$ generate $E'[a]$.
    We have
    \begin{eqnarray*}
        \orientdef{x}(\varphi(P')) &=& \varphi\circ\iota(x)(P) = 0_{E'}\quad (\because \iota(x)(P) \in E[\ia].),\\
        \orientdef{x}(\varphi(Q)) &=& \varphi\circ\iota(x)(0_E) = 0_{E'}.
    \end{eqnarray*}
    Therefore, $E'[a] \subseteq \ker(\orientdef{x})$.
    This means $\orientdef{x} \in a\End(E')$, i.e., $\iota'(x) \in \End(E')$.
    So, we conclude that $\Order \subseteq \Order'$.

    Assume that $\ia$ is an invertible ideal of $\Order$.
    Then $a\ia^{-1}$ is an integral ideal of $\Order$.
    We show $\ker(\hat{\varphi}) = E'[a\ia^{-1}\Order']$.
    From this and the first assertion, we have $\Order' = \Order$.

    We note that $\ker(\hat{\varphi})$ is generated by $\varphi(Q)$.
    Let $x \in \ia^{-1}$.
    For $\alpha \in \ia$, we have
    $$\iota(\alpha ax)Q = \iota(\alpha x)aQ = 0_E.$$
    Therefore, $\iota(ax)Q \in E[\ia] = \ker(\varphi)$.
    So we have $\iota'(ax)\varphi(Q) = \varphi(\iota(ax)Q) = 0_{E'}$.
    This means $\ker(\hat{\varphi}) \subseteq E'[a\ia^{-1}\Order']$.

    Conversely, let $R \in E'[a\ia^{-1}\Order']$.
    Since a non-constant isogeny is surjective,
    there exists $S \in E$ such that $R = \varphi(S)$.
    For $x \in \ia^{-1}$, we have
    $$\varphi(\iota(ax)S) = \iota'(ax)\varphi(S) = \iota'(ax)R = 0_{E'}.$$
    Let $\alpha_1, \dots, \alpha_n \in \ia$ and
    $x_1, \dots, x_n \in \ia^{-1}$ such that
    $\sum_i \alpha_i x_i = 1$.
    Then
    \begin{equation*}
        aS = \iota(\sum_i \alpha_i x_i)aS = \sum_i \iota(\alpha_i)\iota(ax_i)S = 0_E.
    \end{equation*}
    Therefore, we have $\hat{\varphi}(R) = \hat{\varphi}\circ\varphi(S) = aS = 0_E$,
    i.e., $R \in \ker(\hat{\varphi})$.
    So we obtain $E'[a\ia^{-1}\Order'] \subseteq \ker(\hat{\varphi})$.
\end{proof}

Now, we can define an action of $\cl(\Order)$ on $\SSprm$.

\begin{prop}\label{prop:action}
    For $(E, \iota) \in \SSprm$ and an invertible integral ideal $\ia$ of $\Order$ prime to $p$,
    the map
    $(\ia, (E, \iota)) \mapsto \ia * (E, \iota)$
    defines an action of $\cl(\Order)$ on $\SSprm$.
\end{prop}

\begin{proof}
    Let $(E, \iota) \in \SSprm$ and
    $F \in \Ell(\Order)$ such that $(\Emod{F}{\ipe}, \Incmod{\cdotsp}{F}{\ipe}) = (E, \iota)$.
    Then $F[\ia] \coloneqq \cap_{\alpha \in \ia} \ker[\alpha]_F$ 
    corresponds to $E[\ia]$ by the reduction modulo $\ipe$, since $\ia$ is prime to $p$.
    By the complex multiplication theory for elliptic curves over number fields,
    there exist an elliptic curve $F' \in \Ell(\Order)$
    and an isogeny $\varphi: F \to F'$ with $\ker\varphi = F[\ia]$.
    Let $\tilde{\varphi}: \Emod{F}{\ipe} \to \Emod{F'}{\ipe}$ be the reduction of $\varphi$ modulo $\ipe$.
    Then we have 
    $$\ia*(E, \iota) = (\Emod{F'}{\ipe}, \tilde{\varphi}_*(\iota)) = \redmap{\ipe}(F') \in \SSprm.$$

    Let $(E, \iota) \in \SSprm$ and $\ia, \ib$ invertible integral ideals of $\Order$ prime to $p$.
    Write $(E', \iota') = \ia * (E, \iota)$ and $(E'', \iota'') = \ib * (E', \iota')$,
    and let $\varphi_{\ia} : E \to E'$ be a separable isogeny with $\ker\varphi_{\ia} = E[\ia]$
    and $\varphi_{\ib}: E' \to E''$ a separable isogeny with $\ker\varphi_{\ib} = E[\ib]$.
    Then Proposition 3.12 in \cite{Waterhouse69} says
    $\varphi_{\ib}\circ\varphi_{\ia}$ has kernel $E[\ib\ia]$.
    We have
    \begin{eqnarray*}
        \iota'' &=& \frac{1}{\deg\varphi_{\ib}}\varphi_{\ib}\circ\iota'\circ\hat{\varphi_{\ib}}\\
                &=& \frac{1}{\deg\varphi_{\ib}\deg\varphi_{\ia}}\varphi_{\ib}\circ\varphi_{\ia}\circ\iota\circ\widehat{\varphi_{\ia}}\circ\widehat{\varphi_{\ib}}\\
                &=& (\varphi_{\ib}\circ\varphi_{\ia})_*(\iota).
    \end{eqnarray*}
    Therefore, we have
    $$\ib * (\ia * (E, \iota)) = (\ib\ia) * (E, \iota).$$

    It is easy to show that any principal ideal of $\Order$ acts trivially on $\SSprm$.
    Therefore, 
    the map $(\ia, (E, \iota)) \to \ia * (E, \iota)$
    defines an action of $\cl(\Order)$ on $\SSprm$.
\end{proof}

Now, we prove Theorem \ref{thm:group_action}.

\begin{proof}[Proof of Theorem \ref{thm:group_action}]
    It remains to show the action in Proposition \ref{prop:action} is free and transitive.

    Let $\ia$ be an invertible integral ideal of $\Order$ prime to $p$
    such that
    $\ia * (E, \iota) = (E, \iota)$.
    This means that there exists a separable endomorphism $\varphi$ of $E$ with $\ker\varphi = E[\ia]$
    such that
    $\iota = \varphi_*(\iota)$.
    Then $\varphi$ commutes with the endomorphisms in the image of $\iota$.
    Therefore, $\varphi \in \iota(\Order)$.
    Let $\alpha \in \Order$ such that $\varphi = \iota(\alpha)$.
    Since $\varphi$ is separable and $\ia$ is prime to $p$,
    it holds that $\alpha\Order = \ia$.
    Therefore, the action of $\cl(\Order)$ on $\SSprm$ is free.

    By Proposition \ref{prop:num_of_SS} and the assumption on $K$ and $\Order$,
    the set $\SSprm$ is not empty.
    By the definition,
    $\#\SSprm \leq \#\Ell(\Order) = h(\Order)$.
    On the other hand,
    since $\cl(\Order)$ acts freely on $\SSprm$,
    we have $\#\SSprm \geq h(\Order)$.
    Therefore,
    $\#\SSprm = h(\Order)$.
    This shows that the action is transitive.
\end{proof}

\begin{rem}
    CSIDH \cite{CSIDH} uses the set of $\F_p$-isomorphism classes of elliptic curves
    over $\F_p$ whose $\F_p$-endomorphism rings are isomorphic to $\Z[\sqrt{-p}]$.
    Public keys and secret shares in CSIDH is calculated by using
    the free and transitive action of $\cl(\Z[\sqrt{-p}])$ on this set.
    As mentioned in \cite{CK19},
    the group action discussed in this section can be regarded as a generalization
    of the group action in CSIDH.
    We explain this in the following.

    We can define $\Z[\sqrt{-p}]$-orientation on a supersingular elliptic curve over $\F_p$
    by taking $\sqrt{-p}$ to the $p$-th power Frobenius endomorphism.
    Under this orientation,
    an isogeny is $\Z[\sqrt{-p}]$-orientated if and only if it is defined over $\F_p$,
    and two $\Z[\sqrt{-p}]$-orientated elliptic curves are $\Q[\sqrt{-p}]$-isomorphic
    if and only if these curves are $\F_p$-isomorphic.
    Therefore, the set of classes of elliptic curves used in CSIDH is equal to $\SSpr{\Z[\sqrt{-p}]}$,
    and the group action in CSIDH is a special case of Theorem \ref{thm:group_action}.
    Note that, in this case, we have $\SSprm[\Z[\sqrt{-p}]] = \SSpr{\Z[\sqrt{-p}]}$.
\end{rem}
    \section{Orienting supersingular isogeny graphs}\label{sec:graph}
In this section,
we consider a graph related to oriented supersingular elliptic curves.

First, we define an equivalence relation on isogenies.
\begin{defi}
    Two $K$-oriented isogenies
    $$\varphi: (E, \iota_E) \to (F, \iota_F)
    \mbox{ and }
    \psi: (E', \iota_{E'}) \to (F', \iota_{F'})$$
    are {\it $K$-equivalent} if there exist $K$-oriented isomorphisms
    $\lambda : (E, \iota_{E})  \to (E', \iota_{E'})$ and
    $\lambda' : (F', \iota_{F'}) \to (F, \iota_F)$ such that $\lambda'\circ\psi\circ\lambda = \varphi$.
\end{defi}

Let $\ell \neq p$ be a prime number,
and $\Order_\ell^{(0)}$ an order in $K$ such that $\ell$ does not divide the conductor of $\Order_0$.
An {\it $\ell$-isogeny} is an isogeny of degree $\ell$.
We define
$$\Orderl{n}{\ell} \coloneqq \Z + \ell^n\Order_\ell^{(0)},\quad n \in \Z_{\geq 0}.$$

We define a {\it $K$-orienting supersingular $\ell$-isogeny graph}
$\Graph$
as follows:
The vertex set of $\Graph$ is $\bigcup_{n \geq 0}\SSprm[\Orderl{n}{\ell}]$,
and the edges of $\Graph$ are $K$-oriented $\ell$-isogenies up to $K$-equivalence.

Since the reduction map $\redmap{\ipe}$ is bijective into its image,
and an $\ell$-isogeny between $K$-oriented supersingular elliptic curves
corresponds to an $\ell$-isogeny between elliptic curves over a number field,
$\Graph$ has the same structure as that of the $\ell$-isogeny graph of
elliptic curves over a number field with complex multiplication by orders of $K$.
In particular, every $\ell$-isogeny from $\bigcup_{n \geq 0}\SSprm[\Orderl{n}{\ell}]$
has codomain in $\bigcup_{n \geq 0}\SSprm[\Orderl{n}{\ell}]$.
Moreover,
the following analogue of Proposition 23 in \cite{kohel1996} can be obtained by
the graph structure of elliptic curves over a number field with complex multiplication.
\begin{prop}\label{prop:graph_isogeny}
    Let $(E, \iota) \in \SSpr{\Order}$, $D$ be the discriminant of $K$
    and $\Legendre{D}{\ell}$ the Legendre symbol.
    If $\ell$ does not divide the conductor of $\Order$,
    $(E, \iota)$ has
    no ascending $\ell$-isogeny,
    $\Legendre{D}{\ell} + 1$ horizontal $\ell$-isogenies,
    and $\ell - \Legendre{D}{\ell}$ descending $\ell$-isogenies.
    If $\ell$ divides the conductor of $\Order$,
    $(E, \iota)$ has
    exactly one ascending $\ell$-isogeny,
    no horizontal $\ell$-isogeny,
    and $\ell$ descending $\ell$-isogenies.
    Furthermore,
    every codomain of a descending $\ell$-isogeny has exactly
    $[\Order^{\times} : (\Z + \ell\Order)^{\times}]$ $\ell$-isogenies from $(E, \iota)$.
\end{prop}

By this proposition,
the number of primitive $\Orderl{n}{\ell}$-oriented supersingular elliptic curves 
connected to $\SSprm[\Order_\ell{(0)}]$ is
$$\frac{h(\Order_\ell^{(0)})}{[(\Orderl{0}{\ell})^{\times} : (\Orderl{n}{\ell})^{\times}]} \ell^{n - 1}
        \left(\ell - \Legendre{D}{\ell} \right).$$
The formula for the class number of an order (see \S7 in \cite{cox1989primes})
says this number is equal to
$h(\Orderl{n}{\ell}) = \#\SSprm[\Orderl{n}{\ell}]$.
Therefore, all elements in $\SSprm[\Orderl{n}{\ell}]$ have exactly one descending path from
an element in $\SSprm[\Order_\ell^{(0)}]$.
    \section{OSIDH}
Col\`{o} and Kohel \cite{CK19} proposed
a Diffie-Hellman type key-exchange protocol using oriented supersingular elliptic curves,
named OSIDH.
OSIDH uses the action of the ideal class group described in \S\ref{sec:group_action}.
A method to calculate the group action was proposed in \cite{CK19}.
However,
in some cases, it does not work.

In this section, we recall the method to calculate the group action in \cite{CK19}
and the protocol of OSIDH.

\subsection{Group action}\label{sec:group_action_calc}
We recall the method to calculate the group action proposed in \cite{CK19}.

We use the same notation as in Section \ref{sec:graph}.
Let $(E_0, \iota_0) \in \SSprm[\Order_\ell^{(0)}]$.
By Proposition \ref{prop:graph_isogeny}, there is a chain of descending $K$-oriented $\ell$-isogenies:
\begin{equation}\label{eq:lchain}
    \begin{CD}
        (E_0, \iota_0) @>\varphi_0>> (E_1, \iota_1) @>\varphi_1>> (E_2, \iota_2) @>\varphi_2>> \cdots @>\varphi_{n-1}>> (E_n, \iota_n),
    \end{CD}
\end{equation}
where $(E_i, \iota_i) \in \SSprm[\Orderl{i}{\ell}]$ for $i = 0, 1, \dots, n$.
We denote this chain by $((E_i, \iota_i), \varphi_i)$.
Let $q \neq \ell$ be a prime splitting in $K$
and $\iq$ a prime ideal in $\Order_\ell^{(0)}$ lying above $q$.
For brevity, we use the same symbol $\iq$ for the prime ideal $\iq \cap \Orderl{i}{\ell}$ for $i=0, 1, \dots, n$.
Let $(F_i, \iota'_i) = \iq * (E_i, \iota_i)$ and
$\psi_i: (E_i, \iota_i) \to (F_i, \iota'_i)$ be a $K$-oriented isogeny with $\ker\psi_i = E_i[\iq]$.
Then there exists a descending $K$-oriented $\ell$-isogeny
$\varphi'_i: (F_i, \iota'_i) \to (F_{i+1}, \iota'_{i+1})$
with $\ker\varphi'_i = \psi_i(\ker\varphi_i)$.
Therefore, we obtain the following commutative diagram of $K$-oriented isogenies:
\begin{equation}\label{eq:ladder}
    \begin{CD}
        (E_0, \iota_0) @>\varphi_0>> (E_1, \iota_1) @>\varphi_1>> (E_2, \iota_2) @>\varphi_2>> \cdots @>\varphi_{n-1}>> (E_n, \iota_n)\\
        @V\psi_0VV @V\psi_1VV @V\psi_2VV @. @V\psi_{n}VV\\
        (F_0, \iota'_0) @>\varphi'_0>> (F_1, \iota'_1) @>\varphi'_1>> (F_2, \iota'_2) @>\varphi'_2>> \cdots @>\varphi'_{n-1}>> (F_n, \iota'_n).
    \end{CD}
\end{equation}
We denote the chain $((F_i, \iota'_i), \varphi'_i)$ by $\iq * ((E_i, \iota_i), \varphi_i)$.

The method to calculate the group action in \cite{CK19}
is based on the following assumption, though it is not explicitly stated.

\begin{assump}\label{assump:ladder}
    Let $\ell \neq p$ be a prime number,
    $q \neq \ell$ a prime number splitting in $K$,
    $\iq$ a prime ideal in $K$ lying above $q$,
    and $(E, \iota) \to (E', \iota')$ a descending $K$-oriented $\ell$-isogeny.
    We denote $\iq * (E, \iota)$ by $(E'', \iota'')$.
    Let $F$ be an elliptic curve such that
    there exist $q$-isogeny $\psi: E' \to F$ and 
    $\ell$-isogeny $\varphi: E'' \to F$.
    Then $\psi$ and $\varphi$ induce the same orientation on $F$, i.e.,
    $$(F, \psi_*(\iota')) \cong (F, \varphi_*(\iota'')).$$
\end{assump}

Note that this assumption says $\psi$ is horizontal and $\varphi$ is descending.
Therefore,
$(F, \psi_*(\iota'))$ is $K$-isomorphic to $\iq * (E', \iota')$
or $\bar{\iq} * (E', \iota')$,
where $\bar{\iq}$ is the complex conjugate of $\iq$.
Furthermore, if $\iq^2$ is not principal in $\iota(E) \cap \End(E)$ then
$(F, \psi_*(\iota'))$ is $K$-isomorphic to $\iq * (E', \iota')$.

Now we explain the method to the group action stated in \S 5 of \cite{CK19}.
The idea is to compute $j$-invariants of
the chain $\iq * ((E_i, \iota_i), \varphi)$ in the commutative diagram \eqref{eq:ladder}
by using the greatest common divisor of modular polynomials.

For a prime $r$,
we denote the $r$-th modular polynomial by $\Phi_r(X, Y)$.
We assume that the class number of $\Order_\ell^{(0)}$ is one.
Then, in the commutative diagram \eqref{eq:ladder},
it holds that $(F_0, \iota'_0) = (E_0, \iota_0)$.
For simplicity, we assume that $\iq^2$ is not principal in $\Orderl{1}{\ell}$.
First, we need to determine a ``direction'' of $q$-isogeny.
By Assumption \ref{assump:ladder},
an elliptic curve that has an $\ell$-isogeny from $E_0$ and a $q$-isogeny from $E_1$
is $K$-isomorphic to $\iq * (E_1, \iota_1)$ or $\bar{\iq} * (E_1, \iota_1)$.
To distinguish these curves,
we compute the $j$-invariant of $(F_1, \iota'_0) = \iq * (E_1, \iota_1)$
by, for example, V\'elu's formula \cite{Velu1971}.
Then, by Assumption \ref{assump:ladder},
an elliptic curve that has an $\ell$-isogeny from $F_1$ and a $q$-isogeny from $E_2$
is isomorphic to $\iq * (E_2, \iota_2)$.
Therefore, the $j$-invariant of $\iq * (E_2, \iota_2)$ is the unique solution of
$$\gcd(\Phi_\ell(X, j(F_1)), \Phi_q(X, j(E_2))) = 0.$$
We obtain the $j$-invariant of $\iq * (E_2, \iota_2)$
by calculating the g.c.d. of these polynomials.
In the same way, we obtain all the $j$-invariants of the chain $\iq * ((E_i, \iota_i), \varphi_i)$.
By repeating this process, we can obtain 
the $j$-invariants of the chain $\iq^e * ((E_i, \iota_i), \varphi_i)$ for an integer $e$.
Note that for $e \geq 2$,
the calculation of $\iq^e * ((E_i, \iota_i), \varphi_i)$
does not need to determine the direction,
because the chain $\bar{\iq} * (\iq^e * ((E_i, \iota_i), \varphi_i))$
is equal to $(\iq^{e-1} * ((E_i, \iota_i), \varphi_i))$
and we have already know the $j$-invariants of this chain.

Let $q_1, \dots, q_t \neq \ell$ be prime numbers splitting in $K$
and $\iq_i$ a prime ideal of $\Order_\ell^{(0)}$ lying above $q_i$ for $i = 1, \dots, t$.
We further assume that $\iq_i^2 \neq \iq_{i'}^{\pm 2}$ in $\cl(\Orderl{n}{\ell})$ for $i \neq i'$.
As the above, we use the same symbol $\iq_i$ for $\iq_i \cap \Orderl{k}{\ell}$ for $k = 0, 1, \dots, n$.
By the method explained in the previous paragraph,
we can calculate the $j$-invariants of chains of $K$-oriented $q_i$-isogenies
$$(E_n, \iota_n) \to \iq_i * (E_n, \iota_n) \to \cdots \to \iq_i^{e_i} * (E_n, \iota_n)$$
for $i = 1, \dots, t$,
where $e_i$ is a positive integer.
As in the previous paragraph,
we can obtain the $j$-invariant of $(\prod_{i=1}^t\iq_i^{e_i}) * (E_n, \iota_n)$
by calculating the g.c.d.'s of modular polynomials.

\subsection{Protocol}\label{sec:protocol}
The protocol of OSIDH is as follows
(see \S 5.2 of \cite{CK19} for more details):

\noindent
{\bf Public data:}
At the first, Alice and Bob publicly share the following system information.
\begin{itemize}
    \item The $j$-invariants of a chain of descending $K$-oriented $\ell$-isogenies
        $$(E_0, \iota_0) \to (E_1, \iota_1) \to \cdots \to (E_n, \iota_n),$$
        where $(E_0, \iota_0)$ is a primitive $\Order_\ell^{(0)}$-oriented supersingular elliptic curve
        with $h(\Order_\ell^{(0)}) = 1$.
    \item Prime numbers $q_1, \dots, q_t$ splitting in $K$,
        and prime ideals $\iq_1, \dots, \iq_t$ of $\Order_\ell^{(0)}$ above $q_1, \dots, q_t$, respectively.
        We assume that $\iq_i^2 \neq \iq_{i'}^{\pm 2}$ in $\cl(\Orderl{n}{\ell})$ for $i \neq i'$.
    \item The $j$ invariant of $\iq_i * (E_{k_i}, \iota_{k_i})$ for $i = 1, \dots, t$,
        where $k_i$ is the smallest integer such that $\iq^2$ is not principal in $\Orderl{k_i}{\ell}$.
\end{itemize}

\noindent
{\bf Secret key:}
A secret key is an integer vector in $[-B, B]^t$, where $B$ is a positive integer.
We let Alice's secret key be $(e_1, \dots, e_t)$
and Bob's secret key $(d_1, \dots, d_t)$.

\noindent
{\bf Public key:}
Alice's public key is
the $j$-invariant of $(F, \iota') = (\prod_{i=1}^t\iq_i^{e_i}) * (E_n, \iota_n)$
and the $j$-invariants of chains
$$(F, \iota') \to \iq_i * (F, \iota') \to \cdots \to \iq_i^B * (F, \iota'),$$
$$(F, \iota') \to \iq_i^{-1} * (F, \iota') \to \cdots \to \iq_i^{-B} * (F, \iota')$$
for $i = 1, \dots, t$.
These values are calulated by the method explained in \S\ref{sec:group_action_calc}.
Similarly, Bob's public key is
the $j$-invariant of $(G, \iota'') = (\prod_{i=1}^t\iq_i^{d_i}) * (E_n, \iota_n)$
and the $j$-invariants of chains
$$(G, \iota'') \to \iq_i * (G, \iota'') \to \cdots \to \iq_i^B * (G, \iota''),$$
$$(G, \iota'') \to \iq_i^{-1} * (G, \iota'') \to \cdots \to \iq_i^{-B} * (G, \iota'')$$
for $i = 1, \dots, t$.

\noindent
{\bf Shared secret:}
By using the method explained in \S\ref{sec:group_action_calc},
Alice calculates the $j$-invariant of $(\prod_{i=1}^t\iq_i^{e_i}) * (G, \iota'')$,
and Bob calculates the $j$-invariant of $(\prod_{i=1}^t\iq_i^{d_i}) * (F, \iota'')$.
These values are equal to the $j$-invariant of 
$(\prod_{i=1}^t\iq_i^{e_i + d_i}) * (E_n, \iota_n)$.
Alice and Bob shares this value as a shared secret.

    \section{Parameter choice on OSIDH}

\subsection{Counter example}
We show that Assumption \ref{assump:ladder} does not hold in general
by giving a counter example.

Let $p = 419$ and $E_0$ be an elliptic curve over $\F_p$ with $j(E_0) = 52 \equiv 1728 \pmod{419}$.
Then $E_0$ is supersingular and has a primitive $\Z[i]$-orientation,
where $i$ is a square root of -1 in $\Comp$.
By calculating a chain of $3$-isogenies from $E_0$,
we obtain a sequence of $j$-invariants of descending $\Q(i)$-oriented isogenies
$(52, 367, 351 + 244a, 180 + 169a)$,
where $a$ is a square root of $-1$ in $\F_{p^2}$.
Let $\iq$ be a prime ideal of $\Q(i)$ lying above $5$.
By applying the method stated in \S\ref{sec:group_action_calc},
we have
$$\iq * (52, 367, 351 + 244a) = (52, 356, 333 + 132a) \mbox{ or } (52, 356, 180).$$
Which equation holds depends on the choice of $\iq$.
We take $\iq$ such that the former equation holds.
By Assumption \ref{assump:ladder},
there is a unique elliptic curve that has a $5$-isogeny from the curve with $j$-invariant $180 + 169a$
and a $3$-isogeny from the curve with $j$-invariant $333 + 132a$.
However, the equation
$$\gcd(\Phi_5(X, 180 + 169a), \Phi_3(X, 333 + 132a)) = 0$$
has two solutions $202 + 26a$ and $315 + 162a$ in $\overline{\F}_p$.
I.e., in the following diagram, the most right part of the bottom chain
cannot be determined by the modular polynomials.
$$\begin{CD}
    52 @>3>> 367 @>3>> 351 + 244a @>3>> 180 + 169a\\
    @V5VV @V5VV @V5VV @V5VV\\
    52 @>3>> 356 @>3>> 333 + 132a @>3>> 202 + 26a \mbox{ or } 315 + 162a\\
\end{CD}$$
This contradicts Assumption \ref{assump:ladder}.

\subsection{Conditions on $p$}\label{sec:condition}
We give a sufficient that condition that Assumption \ref{assump:ladder} holds.
First,
we recall the structure of the endomorphism ring of a supersingular elliptic curve.
Let $E$ be a supersingular elliptic curve over $k$.
Then $\End(E)$ is isomorphic to a maximal order in the quaternion algebra $B_{p, \infty}$
over $\Q$ ramified at only $p$ and $\infty$.
An order $\Order$ in an imaginary quadratic field $K$ is {\it optimally embedded}
in a maximal order $\mathfrak{O}$ in $B_{p,\infty}$
if $K$ embeds into $B_{p,\infty}$ and $\mathfrak{O} \cap K = \Order$.
In terms of orientation,
this means that there exists a primitive $\Order$-orientation on a supersingular elliptic curve over $k$.
Kaneko \cite{Kaneko1989} proved the following theorem.
\begin{thm}[Theorem 2' in \cite{Kaneko1989}]\label{thm:Kaneko}
    Suppose that two orders $\Order_1$ and $\Order_2$ in $K$
    are optimally embedded in a maximal order in $B_{p,\infty}$ with different images.
    Then the inequality $D_1D_2 \geq p^2$ holds,
    where $D_1$ and $D_2$ are the discriminants of $\Order_1$ and $\Order_2$, respectively.
\end{thm}
This theorem says that
if a supersingular elliptic curve over $k$ has two distinct $K$-orientation
then $p$ is less than or equal to the bound in the theorem.
By using this theorem,
we obtain a sufficient condition for Assumption \ref{assump:ladder} to hold.
\begin{thm}\label{thm:cond_assump}
    In Assumption \ref{assump:ladder},
    we assume that
    $(E, \iota)$ is a primitive $\Orderl{n-1}{\ell}$-oriented supersingular elliptic curve,
    and that $p > q\ell^{2n}D_0$, where $D_0$ is the discriminant of $\Order_0$.
    Then Assumption \ref{assump:ladder} holds.
\end{thm}

\begin{proof}
    Suppose that Assumption \ref{assump:ladder} does not holds.
    Then $F$ has two $K$-orientations $\psi_*(\iota')$ and $\varphi_*(\iota'')$.
    Let $\Order_1$ and $\Order_2$ be orders in $K$
    such that $\psi_*(\iota')$ is a primitive $\Order_1$-orientation
    and $\varphi_*(\iota'')$ is a primitive $\Order_2$-orientation, respectively.
    The discriminant of $\Order_1$ is $q^{-2}\ell^{2n}D_0$, $\ell^{2n}D_0$ or $q^2\ell^{2n}D_0$
    since $\iota'$ is an $\Orderl{n}{\ell}$-orientation and the degree of $\psi$ is $q$.
    Similarly, the discriminant of $\Order_2$ is $\ell^{2n-4}D_0$, $\ell^{2n-2}D_0$ or $\ell^{2n}D_0$.
    By Theorem \ref{thm:Kaneko}, we have
    $q\ell^{2n}D_0 \geq p$.
    This proves the theorem.
\end{proof}

Next, we give a sufficient condition that
all the $j$-invariants of the elliptic curves in $\SSprm[\Orderl{n}{\ell}]$
are distinct.
The protocol of OSIDH correctly works even if there are two oriented elliptic curves
in $\SSprm[\Orderl{n}{\ell}]$ that have the same $j$-invariant.
But, if the number of $j$-invariants in $\SSprm[\Orderl{n}{\ell}]$ is very small,
then the exhaustive search on these $j$-invariants could find a secret share.

Fortunately,
if $p$ satisfies the condition in Theorem \ref{thm:cond_assump}
then all the $j$-invariant of the elliptic curves in $\SSprm[\Orderl{n}{\ell}]$
are distinct.

\begin{thm}\label{thm:dist_jinv}
    Let $\ell \neq p$ be a prime number,
    $D_0$ the discriminants of $\Order_0$,
    $n$ a positive integer such that $p > \ell^{2n}D_0$,
    and $m_1, m_2 \leq n$ nonnegative integers.
    Let $(E_1, \iota_1) \in \SSprm[\Orderl{m_1}{\ell}]$,
    and $(E_2, \iota_2) \in \SSprm[\Orderl{m_2}{\ell}]$.
    Then
    $(E_1, \iota_1) \cong (E_2, \iota_2)$
    if and only if $j(E_1) = j(E_2)$.
\end{thm}

\begin{proof}
    This follows directly from Theorem \ref{thm:Kaneko}.
\end{proof}

\subsection{Security}
Finally, we discuss parameters of OSIDH
for satisfying a certain security level on a classical computer.
Let $\lambda$ be the security level,
i.e., we require OSIDH to satisfy that
at least $2^{\lambda}$ operations are needed to reveal a shared secret in OSIDH.

As in CSIDH, a meet-in-the-middle attack (see \S7.1 in \cite{CSIDH}) can be applied to OSIDH.
Consider the setting in Section \ref{sec:protocol}.
Let $(e_1, \dots, e_t)$ and $(E_A, \iota_A)$ be Alice's secret key and public key, respectively.
The relation between these keys can be written as
$$\left(\prod_{i=1}^{\lfloor t/2 \rfloor}\iq_i^{e_i}\right) * (E_n, \iota_n)
    = \left(\prod_{i=\lfloor t/2 \rfloor + 1}^t \iq_i^{-e_i}\right) * (E_A, \iota_A).$$
So an attacker can find the secret key
by calculating the $j$-invariants of curves that could appear on both sides in the above equation.
The time complexity and memory usage of this attack are proportional to the square root of the size of
the space of public keys.
Therefore, the order of the ideal class group $\cl(\Orderl{n}{\ell})$
has to be greater than $2^{2\lambda}$.
Unfortunately, this condition is not enough.
We need much larger ideal class group.

Let $\alpha$ be an element in $\Orderl{0}{\ell}$ such that $\Orderl{0}{\ell} = \Z + \alpha\Z$.
Then $\ell^n\alpha$ is in $\Orderl{n}{\ell}$.
If one can compute actions of any ideal classes in $\cl(\Orderl{n}{\ell})$
then s/he can compute the endomorphism $\iota_A(\ell^n\alpha)$ on $E_A$.
The endomorphism $\iota_A(\ell^n\alpha)$ is the composition
$$\begin{CD}
    E_A @>\hat{\varphi}>> E_0 @>\iota_0(\alpha)>> E_0 @>\varphi>> E_A,
\end{CD}$$
where $\varphi$ is a descending $\Q(\alpha)$-oriented $\ell^n$-isogeny.
Therefore, by computing $\iota_A(\ell^n\alpha)$,
one can obtain the descending chain from $(E_0, \iota_0)$ to $(E_A, \iota_A)$.
This reveals the secret key. See \S5.1 in \cite{CK19} for the details.

More generally,
the image under $\iota_A$ of a nonnegative element $\beta$ in $\Orderl{n}{\ell}$ reveals the secret.
Since $\beta$ is written as $a + b\ell^n\alpha$, where $a, b \in \Z$ and $b \neq 0$,
one can know the action of $\iota_A(b\ell^n\alpha)$ and the secret descending chain.
(This is an analogy of Petit's attack to SIDH \cite{P}.)

Let $\beta$ be a noninteger element in $\Orderl{n}{\ell}$
and $N$ the norm of $\beta$.
Assume that there exist an integer vector $(e_1, \dots, e_n) \in [-B, B]^n$ such that
$I \coloneqq \prod_i^n \iq_i^{e_i}$ divides $\beta\Orderl{n}{\ell}$ and
the norm $N'$ of the quotient $\beta\Orderl{n}{\ell} / I$ is smooth.
Then there exists horizontal $N'$-isogeny from $(E_A, \iota_A)$ to $I * (E_A, \iota_A)$.
This isogeny can be found by using a meet-in-the-middle attack with a time complexity of $O(\sqrt{N'})$.

For preventing this attack,
we have to choose primes $q_j$ and the range $[-B, B]$ of exponents
such that ideals of form $\prod_{j=1}^t\iq_j^{e_j}$, $e_j \in [-B, B]$
are sufficiently separated from ideals generated by an non-integer element in $\Orderl{n}{\ell}$.
If we take $B$ that satisfies
\begin{equation}\label{eq:cond_principal}
    N_{\mathrm{min}} / \prod_{j=1}^tq_j^B \geq 2^{2\lambda},
\end{equation}
where $N_{\mathrm{min}}$ is the minimum norm of noninteger element in $\Orderl{n}{\ell}$,
then the time complexity of the above attack is greater than $2^\lambda$.
Note that $N_{\mathrm{min}}$ is proportional to $\ell^{2n}$.

As a result, we have to use a proper subset of $\cl(\Orderl{n}{\ell})$ instead of the whole group.
To avoid the meet-in-the-middle attack to search the isogeny $E_n \to E_A$,
we have to choose $t$ and $B$ so that the cardinality of the subset is greater than $2^{2\lambda}$,
i.e.,
\begin{equation}\label{eq:cond_cl}
    (2B + 1)^t \geq 2^{2\lambda}.
\end{equation}

By Theorem \ref{thm:cond_assump},
if $p$ satisfies
\begin{equation}\label{eq:cond_p}
    p > \max_j\{q_j\} \ell^{2n}D_0,
\end{equation}
then we can calculate the group action by the method in \S\ref{sec:group_action_calc}.

Consequently,
if we take $p, \ell, n, \{q_j\}$ and $B$ so that these satisfy
the conditions \eqref{eq:cond_principal}, \eqref{eq:cond_cl} and \eqref{eq:cond_p},
then we can expect that the protocol has a security level of $\lambda$ bits.
These constraints make $p$ very large.

For example, let $K = \Q(\sqrt{-1})$, $\lambda = 128$, $\ell = 2$,
and $\{q_j\}$ be the smallest 100 primes splitting in $\Q(\sqrt{-1})$.
In this case, $D_0 = 4$, $N_{\mathrm{min}} = \ell^{2n}$ and $t = 100$.
Then we have to take $B = 3$, $n = 1428$,
and $p$ greater than 2800 bits.

On the other hand,
SIDH and CSIDH for the same security level use a finite field of characteristic about 500 bits.

\subsection{Future works}
The inequality $p > q\ell^{2n}D_0$
in Theorem \ref{thm:cond_assump} is only a sufficient condition but not a necessary condition.
Further research on the structure of graphs of oriented supersingular elliptic curves
could find a smaller lower bound on $p$.
Furthermore, the inequality \eqref{eq:cond_principal} could also be refined,
since our discussion did not take
the explicit structure of $\cl(\Orderl{n}{\ell})$ into account.
These could make OSIDH more practical.
We leave open these problems for future works.

    \section*{Acknowledgment}
    The author would like to thank the anonymous reviewers for their useful feedback.
    In particular, one reviewer suggested that Section \ref{sec:condition} of a previous version
    of this paper was a direct consequence of a theorem in \cite{Kaneko1989}.

    \bibliographystyle{myplain}
    \bibliography{isog,elliptic_curve}
\end{document}